\pdfoutput=1
\documentclass[a4paper,11pt,oneside]{amsart}
\usepackage[paperheight=279mm,paperwidth=18cm,textheight=26cm,textwidth=14cm,includehead]{geometry}
\usepackage{amsfonts,amssymb,amsmath,amsthm}
\usepackage[T1]{fontenc}
\usepackage[utf8]{inputenc}
\usepackage{mathtools}
\usepackage[unicode]{hyperref}
\hypersetup{
bookmarks=true,
colorlinks=true,
citecolor=[rgb]{0,0,0.5},
linkcolor=[rgb]{0,0,0.5},
urlcolor=[rgb]{0,0,0.75},
pdfpagemode=UseNone,
pdfstartview=FitH,
pdfdisplaydoctitle=true,
pdftitle={Jump inequalities via real interpolation},
pdfauthor={Mirek, Stein, Zorin-Kranich},
pdflang=en-US
}

\usepackage[style=alphabetic,maxalphanames=4]{biblatex}
\numberwithin{equation}{section}
\newtheorem{theorem}[equation]{Theorem}
\newtheorem{proposition}[equation]{Proposition}
\newtheorem{lemma}[equation]{Lemma}
\newtheorem{corollary}[equation]{Corollary}
\theoremstyle{definition}

\theoremstyle{remark}

\newcommand{\R}{\mathbb{R}}
\newcommand{\N}{\mathbb{N}}
\newcommand{\Z}{\mathbb{Z}}
\newcommand{\E}{\mathbb{E}}

\newcommand{\dif}{\mathrm{d}}
\let\rho\varrho
\let\epsilon\varepsilon

\renewcommand{\C}{\mathbb{C}}
\newcommand{\calB}{\mathcal{B}}
\newcommand{\calE}{\mathcal{E}}
\newcommand{\calG}{\mathcal{G}}
\newcommand{\calR}{\mathcal{R}}
\newcommand{\calJ}{\mathcal{J}}
\newcommand{\calN}{\mathcal{N}}
\newcommand{\bbI}{\mathbb{I}}
\newcommand{\FT}{\mathcal{F}} 
\newcommand{\frakm}{\mathfrak{m}}
\newcommand{\frakf}{\mathfrak{f}}
\DeclareMathOperator{\id}{id}
\DeclareMathOperator{\supp}{supp}

\DeclarePairedDelimiter\abs{\lvert}{\rvert}

\DeclarePairedDelimiter\norm{\lVert}{\rVert}
\DeclarePairedDelimiter\card{\lvert}{\rvert}
\providecommand\given{}
\newcommand\SetSymbol[1][]{%
\nonscript\:#1\vert
\allowbreak
\nonscript\:
\mathopen{}}
\DeclarePairedDelimiterX\Set[1]\{\}{\renewcommand\given{\SetSymbol[\delimsize]}#1}
\DeclarePairedDelimiterXPP\EE[1]{\E}{\lparen}{\rparen}{}{\renewcommand\given{\SetSymbol[\delimsize]}#1} 

\newcommand{\one}{\mathbf{1}}
\newcommand{\dis}{\mathrm{dis}}

\newcommand{\vertiii}[1]{{\left\lvert\kern-0.25ex\left\lvert\kern-0.25ex\left\lvert #1 
    \right\rvert\kern-0.25ex\right\rvert\kern-0.25ex\right\rvert}} 

\begin{document}
\title{Jump inequalities via real interpolation}
\author{Mariusz Mirek}
\address[Mariusz Mirek]{
  Department of Mathematics,
  Rutgers University,
Piscataway, NJ 08854, USA \&
Instytut Matematyczny,
Uniwersytet Wroc{\l}awski,
Plac Grunwaldzki 2/4,
50-384 Wroc{\l}aw
Poland}
\email{mirek@math.uni.wroc.pl}

\author{Elias M. Stein}
\address[Elias M. Stein]{
Department of Mathematics,
Princeton University,
Princeton,
NJ 08544-100 USA}
\email{stein@math.princeton.edu}

\author{Pavel Zorin-Kranich}
\address[Pavel Zorin-Kranich]{
  Mathematical Institute,
  University of Bonn,
  Endenicher Allee 60,
  53115 Bonn,
  Germany}
\email{pzorin@math.uni-bonn.de}

\begin{abstract}
Jump inequalities are the $r=2$ endpoint of L\'epingle's inequality for $r$-variation of martingales.
Extending earlier work by Pisier and Xu \cite{MR933985} we interpret these inequalities in terms of Banach spaces which are real interpolation spaces.
This interpretation is used to prove endpoint jump estimates for vector-valued martingales and doubly stochastic operators as well as to pass via sampling from $\R^{d}$ to $\Z^{d}$ for jump estimates for Fourier multipliers.
\end{abstract}

\thanks{Mariusz Mirek was partially supported by the Schmidt Fellowship and the IAS Found for Math.\ and by the National Science Center, NCN grant DEC-2015/19/B/ST1/01149.
  Elias M.~Stein was partially supported by NSF grant DMS-1265524.
  Pavel Zorin-Kranich was partially supported by the Hausdorff Center for Mathematics and DFG SFB-1060.}

\maketitle

\section{Introduction}
L\'epingle's inequality \cite{MR0420837} is a refinement of Doob's martingale maximal inequality in which the supremum over the time parameter is replaced by the stronger $r$-variation norm, $r>2$ (see Section~\ref{sec:r-variation} for the definition).
Since a sequence with bounded $r$-variation norm is necessarily a Cauchy sequence, L\'epingle's inequality is also a quantitative form of the martingale convergence theorem.

We are interested in the $r=2$ endpoint that had been first stated by Pisier and Xu \cite{MR933985} on $L^{2}$ and in a different formulation by Bourgain \cite[inequality  (3.5)]{MR1019960} on $L^{p}$, $1<p<\infty$.
We begin with Bourgain's formulation.

Throughout the article $B$ denotes a Banach space and $\bbI$ is a totally ordered set.
Unless otherwise stated, we consider only finite totally ordered sets $\bbI$.
This will ensure measurability of all functions that we define. All our estimates do not depend on
the cardinality of $\bbI$ and  the passage to the limiting case of infinite index sets $\bbI$ will be permitted by the monotone convergence theorem.

For any $\lambda>0$ the \emph{$\lambda$-jump counting function} of a function $f : \bbI \to B$ is defined by
\begin{align}
  \label{eq:def:jump}
  \begin{split}
N_{\lambda}(f)& :=
N_{\lambda}(f(t) : t\in\bbI)\\
&:=\sup \Set{J\in\N \given \exists_{\substack{t_{0}<\dotsb<t_{J}\\ t_{j}\in\bbI}}  : \min_{0<j\leq J} \norm{f(t_{j})-f(t_{j-1})}_{B} \geq \lambda}.  
  \end{split}
\end{align}
The quantity $N_{\lambda}(f(t) : t\in\bbI)$ is monotonically increasing in $\bbI$ and does not change upon replacing $f(t)$ by $f(t)+b$ for any fixed $b\in B$.

For a measure space $(X,\calB,\frakm)$ we consider the family of \emph{jump quasi-seminorms} on functions $f : X \times \bbI \to B$ defined by
\begin{align}
  \label{eq:jump-space}
  \begin{split}
  J^{p,q}_{\rho}(f)&:=J^{p,q}_{\rho}(f : X \times \bbI \to B):=J^{p,q}_{\rho}((f(\cdot, t))_{t\in\bbI} : X  \to B) \\
  &:=\sup_{\lambda>0} \norm[\big]{ \lambda N_{\lambda}(f(\cdot, t): t\in \bbI)^{1/\rho} }_{L^{p,q}(X)}
  \end{split}
\end{align}
for $0<p<\infty$, $0<q\leq\infty$, and $0<\rho<\infty$, where $L^{p,q}(X)$ denotes Lorentz spaces, see e.g.~\cite[Section 1.4]{MR3243734}.
The quantity $J^{p,q}_{\rho}$ is monotonically decreasing in $q$ and in $\rho$.
As for Lebesgue spaces $L^{p}(X)=L^{p,p}(X)$ we omit the index $q$ if $q=p$.

The jump quasi-seminorms $J^{p}_{r}$ are dominated by $L^p$ norms of $r$-variations (which are recalled in \eqref{eq:192}) in view of \eqref{eq:195}.
On the other hand, $L^{p}$ bounds for $r$-variations can be deduced from estimates for $J^{p}_{\rho}$ only when $r>\rho$, using Lemma~\ref{lem:jumps-control-variation} and interpolation.
Since in many situations one does not have $r$-variational control for $r=2$, having jump control for $\rho=2$ is then an end-point refinement.

We now briefly highlight the main results of this paper.
\begin{enumerate}
\item We prove that the quantities $J^{p,q}_{\rho}$ are in fact equivalent to Banach space norms which arise via real interpolation (see Lemma~\ref{lem:pisier-xu-space} and Corollary~\ref{cor:pisier-xu-space-convex}).
\item We obtain the end-point versions (for $\rho=2$) of L\'epingle's inequality, in the vector-valued setting.
  These are given in Theorems~\ref{thm:endpoint-lepingle:simple} and Theorem~\ref{thm:endpoint-lepingle}.
\item We show (in Theorem~\ref{thm:markov-jump}) that the end-point $\rho=2$ results hold for $J^{p}_{\rho}$ for doubly stochastic operators (hence for symmetric diffusion semi-groups), while previously these results were known only for $\rho>2$.
\item We also have (in Theorem~\ref{thm:msw-mult-jumps}) an extension of the sampling technique, that allows us to deduce jump inequalities in the discrete case from the continuous case, for appropriate Fourier multipliers.
  The original sampling theorem (in \cite{MR1888798}) was limited to $L^{p}(B)$, while the jump quantities are not equivalent to Banach spaces of this type.
\end{enumerate}
The results sketched above are applied in companion articles \cite{arxiv:1808.09048,arxiv:1809.03803}.

In more detail, one of our main results is the endpoint L\'epingle inequality for vector-valued martingales.
\begin{theorem}
\label{thm:endpoint-lepingle:simple}
Let  $p\in(1,\infty)$ and let $B$ be a Banach space with martingale cotype $\rho\in[2, \infty)$.
Then for every $\sigma$-finite measure space $(X,\calB,\frakm)$, every finite totally ordered set $\bbI$, and every martingale $\frakf=(\frakf_{t})_{t\in \bbI} : X \to B$ indexed by $\bbI$ with values in $B$ we have
\begin{equation}
\label{eq:endpoint-lepingle:simple}
J^{p}_{\rho}(\frakf : X \times \bbI \to B)
\lesssim_{p,\rho,B}
\sup_{t\in \bbI}\norm{\frakf_{t}}_{L^p(X;B)},
\end{equation}
where the implicit constant does not depend on the martingale $\frakf$.
\end{theorem}
The notion of cotype for Banach spaces is recalled in Section~\ref{sec:martingale}, where a more precise version of Theorem~\ref{thm:endpoint-lepingle:simple}, namely Theorem~\ref{thm:endpoint-lepingle}, is stated and proved.
Hilbert spaces, and in particular the scalar field $B=\C$, have martingale cotype $2$.

In the above formulation, the scalar case $B=\C$ of Theorem~\ref{thm:endpoint-lepingle:simple} is due to Bourgain \cite[inequality  (3.5)]{MR1019960}. In the vector-valued case Theorem~\ref{thm:endpoint-lepingle:simple} is an endpoint of the $r$-variation, $r>\rho$, estimate in \cite[Theorem 4.2]{MR933985}.
The basic argument that deduces the $r$-variation from $\rho$-jump estimates appears in \cite[Lemma 2.1]{MR2434308}.
We record a refined version of that argument in Lemma~\ref{lem:jumps-control-variation}.

Prior to \cite{MR1019960} an endpoint L\'epingle inequality in the case $p=\rho=2$, $B=\C$, formulated in terms of an estimate in a real interpolation space, appeared in an article by Pisier and Xu \cite[Lemma 2.2]{MR933985}.
Our Lemma~\ref{lem:pisier-xu-space} shows that the jump quasi-seminorms \eqref{eq:jump-space} are equivalent to the (quasi-)norms on certain real interpolation spaces which include those used by Pisier and Xu.
In particular this shows that the endpoint L\'epingle inequalities of Bourgain and of Pisier and Xu are equivalent.
While the formulation in terms of the jump quasi-seminorm is convenient for some purposes (e.g.\ in the proof of Lemma~\ref{lem:jumps-control-variation}), the real interpolation point of view, further explained in Section~\ref{sec:jump-interpolation}, turns out to be crucial in the proofs of Theorem~\ref{thm:markov-jump} and Theorem~\ref{thm:msw-mult-jumps} below.

By Rota's dilation theorem estimates for martingales can be transferred to doubly stochastic operators.
In the case of jump inequalities an additional complication arises.
Namely, the quasi-seminorm \eqref{eq:jump-space} does not seem to define a vector-valued Lorentz space, so it appears unclear whether conditional expectation operators are bounded with respect to it.
This obstacle will be overcome by Lemma~\ref{lem:pisier-xu-space} and the Marcinkiewicz interpolation theorem.
This allows us to deduce the following result.

\begin{theorem}
\label{thm:markov-jump}
Let $(X,\calB,\frakm)$ be a $\sigma$-finite measure space and let $Q$ be a doubly stochastic operator on $L^{1}(X) + L^{\infty}(X)$, that is,
\begin{enumerate}
\item $\norm{Qf}_{L^1(X)} \leq \norm{f}_{L^1(X)}$ for all $f\in L^{1}(X)$,
\item $\norm{Qf}_{L^{\infty}(X)} \leq \norm{f}_{L^{\infty}(X)}$ for all $f\in L^{\infty}(X)$,
\item $f\geq 0 \implies Qf \geq 0$,
\item $Q \one_{X} = Q^{*} \one_{X} = \one_{X}$.
\end{enumerate}
Let $B$ be a Banach space with martingale cotype $\rho \in [2,\infty)$.
Then for every $p \in (1,\infty)$ and every measurable function $f : X \to B$ we have
\begin{equation}
\label{eq:markov-jump}
J^{p}_{\rho}(( (Q^{*})^{n} Q^{n}f)_{n\in\N} : X  \to B)
\lesssim_{p,\rho,B}
\norm{f}_{L^{p}(X;B)}.
\end{equation}
\end{theorem}
Here we identify $Q$ and $Q^{*}$ with the unique contractive extensions of the tensor products $Q \otimes \id_{B}$ and $Q^{*} \otimes \id_{B}$ to $L^{p}(X;B)$, respectively.

Theorem~\ref{thm:markov-jump} is new even in the case $B=\C$ providing the endpoint to \cite[Proposition 3.1(1)]{MR1869071}.
In the vector-valued case Theorem~\ref{thm:markov-jump} is an endpoint of \cite[Theorem 5.2]{MR3648493} (with $m=0$) and \cite[Corollary 6.2]{MR3648493}.
In \cite{MR3648493} the non-endpoint result was formulated for Banach spaces that are interpolation spaces between a Hilbert space and a uniformly convex space.
It was recently extended to general Banach spaces with finite cotype in \cite{arxiv:1803.05107}.

Applying Theorem~\ref{thm:markov-jump} with $B=\C$ and $\rho=2$ one can obtain endpoint jump inequalities in a number of situations, e.g.\ for averages associated to convex bodies (see \cite{MR3777413}; we give a simplified proof in \cite{arxiv:1808.09048}) or to spheres in the free group (using Bufetov's proof \cite{MR1923970,MR2923460} of the result from \cite{MR1294672} and its extensions).

Our last result concerns periodic Fourier multipliers for functions on $\Z^{d}$.
From \cite[Section 2]{MR1888798} we know how sampling can be used to pass from $\R^{d}$ to $\Z^{d}$ in Bochner space multiplier estimates, see also Proposition~\ref{prop:msw-mult} (a Bochner space is a vector-valued space $L^{p}(B)$, where $B$ is a Banach space).
These results do not apply to the spaces defined by \eqref{eq:jump-space} because they are not Bochner spaces.
In Proposition~\ref{prop:msw-mult-interpolation} we extend this sampling technique to real interpolation spaces between Bochner spaces (that are in general not Bochner spaces, see \cite{MR0358326} for counterexamples).
This leads to the following result on jump spaces.
\begin{theorem}
\label{thm:msw-mult-jumps}
Let $p,\rho \in (1,\infty)$.
Let $q\in\Z$ be a positive integer, $\bbI$ a countable ordered set, and $m : \R^{d} \to \C^{\bbI}$ a bounded sequence-valued measurable function supported on $\frac{1}{q}[-1/2,1/2]^{d} \times \bbI$.
Let $T$ be the sequence-valued Fourier multiplier operator corresponding to $m$.
Define a periodic multiplier by
\begin{equation}
\label{eq:m_per}
m^{q}_{\mathrm{per}}(\xi) : = \sum_{l\in\Z^{d}} m(\xi-l/q)
\end{equation}
and denote the associated Fourier multiplier operator over $\Z^{d}$ by $T^{q}_{\dis}$.
Then
\[
\norm{ T^{q}_{\dis} }_{\ell^{p}(\Z^{d}) \to J^{p}_{\rho}(\Z^{d} \times \bbI \to \C)}
\lesssim_{p,\rho,d}
\norm{ T }_{L^{p}(\R^{d}) \to J^{p}_{\rho}(\R^{d} \times \bbI \to \C)}.
\]
\end{theorem}

Several multipliers to which Theorem~\ref{thm:msw-mult-jumps} applies can be found in the article \cite{MR2434308}.
Using Theorem~\ref{thm:msw-mult-jumps} instead of \cite[Corollary 2.1]{MR1888798} one can obtain the $r=2$ jump endpoint of the variational estimates in \cite{MR3681393}.
Details of this argument appear in \cite{arxiv:1809.03803}.

We follow the convention that $x \lesssim y$ ($x \gtrsim y$) means that $x \leq C y$ ($x \geq C y$) for some absolute constant $C>0$.
If $x \lesssim y$ and $x \gtrsim y$ hold simultaneously then we will write $x \simeq y$. 
The dependence of the implicit constants on some parameters is indicated by a subscript to the symbols $\lesssim$, $\gtrsim$ and $\simeq$. 

The set of non-negative integers will be denoted by $\N:=\{0, 1, 2, \ldots\}$.

\section{Jump inequalities: abstract theory}
\label{sec:jump-interpolation}

\subsection{$r$-variation}
\label{sec:r-variation}
The \emph{$r$-variation (quasi-)seminorm} of a function $f : \bbI \to B$ is defined by
\begin{align}
  \label{eq:192}
  \begin{split}
V^{r}(f)
:=&V^{r}(f(t) : t\in\bbI)\\
:=&
\begin{cases}
\sup_{J\in\N} \sup_{\substack{t_{0}<\dotsb<t_{J}\\ t_{j}\in\bbI}}\Big(\sum_{j=0}^{J-1}  \norm{f(t_{j+1})-f(t_{j})}_B^{r} \Big)^{1/r},  
&
0< r <\infty,\\
\sup_{\substack{t_{0}<t_{1}\\ t_{j}\in\bbI}} \norm{f(t_{1})-f(t_{0})}_B,  
& r = \infty,
\end{cases}
\end{split}
\end{align}
where the former supremum is taken over all finite increasing sequences in $\bbI$.

The quantity $V^{r}(f(t) : t\in\bbI)$ is monotonically increasing in $\bbI$ and monotonically decreasing in $r$.
Moreover,
\begin{align}
\label{eq:157}
V^r(f(t): t\in \bbI)
\leq
2 \Big(\sum_{j\in\bbI}  \norm{f(j)}_B^{r} \Big)^{1/r}
\quad
\text{for } 1\leq r<\infty.
\end{align}
The quantity $V^{r}(f)$ vanishes if and only if the function $f : \bbI \to B$ is constant.
For the purpose of using interpolation theory it is convenient to factor out the constant functions.
The space of sequences $f : \bbI \to B$ with bounded $r$-variation modulo the constant functions is denoted by
\[
V^{r}_{\bbI \to B} := \Set{ f : \bbI \to B \given V^{r}(f) < \infty } / B.
\]
For every $\lambda>0$ and $0<r<\infty$ we have
\begin{align}
\label{eq:195}
\lambda N_{\lambda}(f(t): t\in \bbI)^{1/r}\le V^r(f(t): t\in \bbI).
\end{align}
Indeed, for any increasing sequence $t_{0}<\dotsb<t_{J}$ in $\bbI$ with $\norm{f(t_{j})-f(t_{j-1})}_{B} \geq \lambda$ for all $0<j\leq J$ as in the definition \eqref{eq:def:jump} of $N_{\lambda}(f)$ we have
\[
\lambda J^{1/r}
=
\bigl( \sum_{j=1}^{J} \lambda^{r} \bigr)^{1/r}
\leq
\bigl( \sum_{j=1}^{J} \norm{f(t_{j})-f(t_{j-1})}_{B}^{r} \bigr)^{1/r}
\leq
V^r(f(t): t\in \bbI).
\]


\subsection{Real interpolation}
In this section we recall the definition of Peetre's $K$-method of real interpolation.
The classical reference on this subject is the book \cite{MR0482275}.

A pair $\bar A = (A_{0},A_{1})$ of quasinormed vector spaces is called \emph{compatible} if they are both contained in some ambient topological vector space and the intersection $A_{0}\cap A_{1}$ is dense both in $A_{0}$ and in $A_{1}$.
For any $a\in A_{0}+A_{1}$ and $t\in [0, \infty)$ the \emph{$K$-functional} is defined by
\[
K(t,a;\bar A) := \inf_{a=a_{0}+a_{1}} (\norm{a_{0}}_{A_{0}} + t \norm{a_{1}}_{A_{1}}).
\]
The function $t\mapsto K(t,a;\bar A)$ is non-negative and non-decreasing on $[0, \infty)$.
It is concave if the quasinorms on $A_{0}$ and $A_{1}$ are subadditive (that is, actual norms).
Also, for $s\in (0,\infty)$ and $t\in [0,\infty)$ we have
\begin{equation}
\label{eq:206}
K(t,a; A_{0}, A_{1}) \leq \max(1,t/s) K(s,a; A_{0}, A_{1}).
\end{equation}

The real interpolation space $[A_{0},A_{1}]_{\theta,r}$ for $0<\theta<1$ and $0< r < \infty$ is defined by the quasi-norm
\begin{equation}
\label{eq:KJ-norms}
[A_{0},A_{1}]_{\theta,r}(a)
:=
\Big(\sum_{j\in\Z}\big(2^{-j\theta}K(2^{j},a;\bar A)\big)^r\Big)^{1/r}
\end{equation}
with the natural modification
\begin{equation}
\label{eq:KJ-inf-norms}
[A_{0},A_{1}]_{\theta,\infty}(a)
:=
\sup_{j\in\Z} \big(2^{-j\theta}K(2^{j},a;\bar A)\big)
\end{equation}
in the case $r=\infty$.
If quasinorms on both spaces $A_{0},A_{1}$ are in fact norms and $1\leq r\leq\infty$, then \eqref{eq:KJ-norms} defines a norm on $[A_{0},A_{1}]_{\theta,r}$.
From monotonicity of $\ell^{r}$ norms it is easy to see that
\[
[A_{0},A_{1}]_{\theta,r}
\subseteq
[A_{0},A_{1}]_{\theta,q},
\quad\text{ whenever }\quad  0<r\leq q\leq\infty.
\]
Note that
\[
K(t,a; A_0, A_1)=tK(t^{-1},a; A_1, A_0),
\]
so $[A_{0},A_{1}]_{\theta,r}=[A_{1},A_{0}]_{1-\theta,r}$.

\subsection{Jump inequalities as an interpolation space}

The following observation seems to be new and allows us to use standard real interpolation tools to deal with jump inequalities.
\begin{lemma}
\label{lem:pisier-xu-space}
For every $p\in(0, \infty)$, $q\in(0,\infty]$, $\rho\in(0, \infty)$, and $\theta\in(0, 1)$ there exists a constant $0<C=C_{p, q, \rho, \theta} < \infty$ such that the following holds.
Let $(X, \calB, \frakm)$ be a measure space, $\bbI$ a finite totally ordered set, $B$ a Banach space, and $f : X \times \bbI \to  B$ a measurable function.
Then
\begin{equation}
\label{eq:197}
C^{-1} J^{p,q}_{\rho}(f)
\le
\big[L^{\infty}(X;V_{\bbI \to B}^{\infty}),L^{\theta p, \theta q}(X;V_{\bbI \to B}^{\theta\rho})\big]_{\theta,\infty}(f)
\le C J^{p,q}_{\rho}(f).
\end{equation}
\end{lemma}
As already mentioned here and later we consider finite totally ordered sets $\bbI$.
This ensures measurability of all functions that we define and allows us to use stopping time arguments.
All our estimates do not depend on the cardinality of $\bbI$ and therefore the passage to the limiting case of infinite index sets $\bbI$ will be permitted.
\begin{proof}
We begin with the first inequality in \eqref{eq:197} and normalize
\[
[L^{\infty}(X;V_{\bbI \to B}^{\infty}),L^{\theta p,\theta q}(X;V_{\bbI \to B}^{\theta\rho})]_{\theta,\infty}(f) = 1.
\]
By definition of real interpolation spaces this is equivalent to
\begin{equation}
\label{eq:K(Linfty(ellinfty),Lp/2(V1))}
K(\lambda^{1/\theta},f;L^{\infty}(X;V_{\bbI \to B}^{\infty}), L^{\theta p,\theta q}(X;V_{\bbI \to B}^{\theta\rho})) \lesssim \lambda
\end{equation}
for all $\lambda>0$.

For a fixed $\lambda>0$ we apply \eqref{eq:K(Linfty(ellinfty),Lp/2(V1))} with $\lambda$ replaced by $c\lambda$ with some small $c=c(p)$.
By definition of the $K$-functional there exists a splitting $f=f^{0}+f^{1}$ with
\begin{align*}
\norm{V^{\infty}(f^{0}(\cdot, t): t\in \bbI)}_{L^{\infty}(X)}
&< \lambda/2,\\
\norm{ V^{\theta\rho}(f^{1}(\cdot, t): t\in \bbI)}_{L^{\theta p,\theta q}(X)}
&\lesssim \lambda^{1-1/\theta}.
\end{align*}
Now, any $\lambda$-jump of $f$ corresponds to a $\lambda/2$-jump of $f^{1}$ in the sense that for every $x\in X$ and any increasing sequence $t_{0}<\dotsb<t_{J}$ in $\bbI$ as in \eqref{eq:def:jump} we have
\begin{align*}
\norm{f^{1}(x,t_{j})-f^{1}(x,t_{j-1})}_B
&\geq
\norm{f(x,t_{j})-f(x,t_{j-1})}_B
-
\norm{f^{0}(x,t_{j})-f^{0}(x,t_{j-1})}_B\\
&\geq
\lambda
-
V^{\infty}(f^{0}(x,t) : t\in\bbI)\\
&\geq
\lambda/2.
\end{align*}
Therefore by \eqref{eq:195} we get
\begin{align*}
N_{\lambda}(f(x, t): t\in \bbI)
&\leq
N_{\lambda/2}(f^{1}(x, t): t\in \bbI)\\
&\lesssim
\lambda^{-\theta\rho} V^{\theta\rho}(f^{1}(x, t): t\in \bbI)^{\theta\rho}.
\end{align*}
It follows that
\[
\norm{ \lambda N_{\lambda}^{1/\rho}(f(\cdot,t) : t \in \bbI) }_{L^{p,q}(X)}
\lesssim
\lambda^{1-\theta} \norm{ V^{\theta\rho}(f^{1}(\cdot,t) : t \in \bbI) }_{L^{\theta p,\theta q}(X)}^{\theta}
\lesssim
1.
\]
This proves the first inequality in \eqref{eq:197}.

We now prove the second inequality in \eqref{eq:197}.
Normalizing
\[
\sup_{\lambda>0} \norm{ \lambda N_{\lambda}^{1/\rho}(f(\cdot,t) : t\in\bbI) }_{L^{p,q}(X)} = 1
\]
we have to show \eqref{eq:K(Linfty(ellinfty),Lp/2(V1))}.
Fix $\lambda>0$ and construct stopping times (measurable functions) $t_{0},t_{1},\dotsc : X \to \bbI \cup \Set{+\infty}$ starting with $t_{0}(x) := \min \bbI$.
Given $t_{k}(x)$, let
\begin{equation}
\label{eq:lambda-jump-stopping-time}
t_{k+1}(x) := \min(\Set{t\in\bbI \given t>t_{k}(x) \text{ and } \norm{f(x,t)-f(x,t_{k}(x))}_{B}\geq\lambda} \cup \Set{+\infty}),
\end{equation}
with the convention that $+\infty$ is greater than every element of $\bbI$.
For $t\in \bbI$ let
\[
k(x,t) := \max \Set{k\in\N \given t_{k}(x) \leq t}.
\]
With this stopping time we split $f=f^{0}+f^{1}$, where
\[
f^{1}(x,t) := f(x,t_{k(x,t)}(x)),
\quad
f^{0}(x,t) := f(x,t) - f(x,t_{k(x,t)}(x)).
\]
By construction of the stopping time we have $\norm{f^{0}(x,t)}_{B}<\lambda$ for all $x\in X$ and $t\in\bbI$, so that
\[
\norm{V^{\infty}(f^{0}(\cdot,t) : t\in\bbI )}_{L^{\infty}(X)} \leq 2\lambda.
\]
On the other hand, $f^{1}(x,t)$ is constant for $t_{k}(x)\leq t<t_{k+1}(x)$, so while estimating its variation norm we can restrict the supremum in \eqref{eq:192} to sequences taking values in the sequence of stopping times $(t_{k}(x))_{k\in\N}$.
With $\alpha=\alpha(p,q,\rho)>1$ to be chosen later we split the jumps according to their size and obtain
\begin{align*}
\MoveEqLeft
V^{\theta\rho}(f^{1}(x, t_k(x)): k\in\N)^{\theta\rho}\\
&=
\sup_{k_0<\ldots<k_J} \sum_{i=1}^{J} \norm{f^1(x, t_{k_{i}}(x))-f^1(x, t_{k_{i-1}}(x))}_B^{\theta\rho}\\
&\leq
\sup_{k_0<\ldots<k_J} \Bigl( (\alpha \lambda)^{\theta\rho} \abs{\Set{i\in\N \given 0 < \norm{f^1(x, t_{k_{i}}(x))-f^1(x, t_{k_{i-1}}(x))}_B \leq \alpha\lambda}}\\
&\qquad+
\sum_{n>0} (\alpha^{n+1} \lambda)^{\theta\rho} \abs{\Set{i\in\N \given \alpha^n\lambda < \norm{f^1(x, t_{k_{i}}(x))-f^1(x, t_{k_{i-1}}(x))}_B \leq \alpha^{n+1}\lambda}} \Bigr)\\
&\leq
(\alpha\lambda)^{\theta\rho} N_{\lambda}(f(x, t): t\in \bbI) + \sum_{n>0} (\alpha^{n+1}\lambda)^{\theta\rho} N_{\alpha^{n}\lambda}(f(x, t): t\in \bbI)\\
&\lesssim
\lambda^{\theta\rho} \sum_{n\geq 0} \alpha^{n\theta\rho} N_{\alpha^{n}\lambda}(f(x, t): t\in \bbI).
\end{align*}
Hence
\begin{align*}
\MoveEqLeft
\norm[\big]{V^{\theta\rho}(f^{1}(\cdot, t_k(\cdot)): k\in\N)}_{L^{\theta p,\theta q}(X)}\\
&\lesssim
\lambda \norm[\Big]{ \Big(\sum_{n\geq 0} \alpha^{n\theta\rho}
N_{\alpha^{n}\lambda}(f(\cdot, t): t\in \bbI)\Big)^{1/(\theta\rho)} }_{L^{\theta p,\theta q}(X)}\\
&=
\lambda^{1-1/\theta} \norm[\Big]{ \sum_{n\geq 0} \alpha^{-n(1-\theta)\rho}
\Big( \underbrace{ \alpha^{n}\lambda N_{\alpha^{n}\lambda}(f(\cdot, t): t\in \bbI)^{1/\rho} }_{(*)} \Big)^{\rho} }_{L^{p/\rho,q/\rho}(X)}^{1/(\theta\rho)}
\end{align*}
By the hypothesis the $L^{p,q}(X)$ norm of the highlighted function $(*)$ is at most $1$, so the $L^{p/\rho,q/\rho}(X)$ norm of its $\rho$-th power is also $\leq 1$.
The series can be summed provided that $\alpha$ is sufficiently large in terms of the quasimetric constant of the scalar-valued Lorentz space $L^{p/\rho,q/\rho}(X)$.
Hence the splitting $f=f^{0}+f^{1}$ witnesses the inequality \eqref{eq:K(Linfty(ellinfty),Lp/2(V1))} for the $K$-functional.
\end{proof}

\begin{corollary}
\label{cor:pisier-xu-space-convex}
For every $p \in (1,\infty)$, $q \in (1,\infty]$, and $\rho \in (1,\infty)$ there exists a constant $0<C<\infty$ such that for every measure space $(X,\calB,\frakm)$, finite totally ordered set $\bbI$, and Banach space $B$ there exists a (subadditive) seminorm $\vertiii{\cdot}$ equivalent to the quasi-seminorm on $J^{p,q}_{\rho}(\cdot)$ in the sense that
\[
C^{-1} \vertiii{f} \leq J^{p,q}_{\rho}(f : X\times\bbI \to B) \leq C \vertiii{f}.
\]
for all measurable functions $f : X \times \bbI \to B$.
\end{corollary}

\begin{proof}
Let $\max(1/p,1/q,1/\rho) < \theta < 1$.
Then the quasinorm $V^{\theta\rho}_{\bbI \to B}$ is a norm, and the vector-valued Lorentz space $L^{\theta p,\theta q}(X; V^{\theta\rho}_{\bbI\to B})$ admits an equivalent norm (with equivalence constants depending only on $\theta p$ and $\theta q$).
Hence the interpolation quasinorm in \eqref{eq:197} is actually a norm.
\end{proof}

\subsection{From jump inequalities to $r$-variation}
It has been known since Bourgain's article \cite{MR1019960} that variational estimates can be deduced from jump inequalities.
In Bourgain's article this is accomplished for averaging operators by interpolation with an $L^{\infty}$ estimate.
More in general, Jones, Seeger, and Wright showed that $L^{p}$ jump inequalities with different values of $p$ can be interpolated to yield variation norm estimates \cite[Lemma 2.1]{MR2434308}.
Our next result is that a weak type jump inequality implies a weak type estimate for the variation seminorm for a fixed $p$.
\begin{lemma}
\label{lem:jumps-control-variation}
For every $p\in[1, \infty]$ and $1\leq\rho<\infty$ there exists $0<C_{p,\rho}<\infty$ such that the following holds.
Let $(X, \calB, \frakm)$ be a $\sigma$-finite measure space, $\bbI$ a finite totally ordered set, $B$ a Banach space, and $\rho < r \leq \infty$.
Then for every measurable function $f : X \times \bbI \to B$ we have the estimates
\begin{equation}
\label{eq:Vr<N}
\norm{f}_{L^{p,\infty}(X; V^{r}_{\bbI \to B})}
\le C_{p,\rho}
\begin{cases}
\big(\frac{r}{r-\rho}\big)^{1/p} J^{p,\infty}_{\rho}(f)
& \text{if } p<\rho,\\
\big(\frac{r}{r-\rho} \big(1 + \log \frac{r}{r-\rho} \big) \big)^{1/\rho} J^{p,\infty}_{\rho}(f)
& \text{if } p=\rho \text{ and}\\
\big(\frac{r}{r-\rho} \big)^{1/\rho} J^{p}_{\rho}(f)
& \text{if } p=\rho,\\
\big(\frac{r}{r-\rho}\big)^{1/\rho} J^{p,\infty}_{\rho}(f)
& \text{if } p>\rho.
\end{cases}
\end{equation}
\end{lemma}
The previous result \cite[Lemma 2.1]{MR2434308} can be recovered by scalar-valued real interpolation.
Moreover, Lemma~\ref{lem:jumps-control-variation} reduces L\'epingle's inequality at the endpoint $p=1$ to the jump inequality \eqref{eq:L1-lepingle}.

In the case $p=\rho$ we will use log-convexity of $L^{1,\infty}$.
More precisely, let $(X,\calB,\frakm)$ be a measure space, $I$ a countable set, and $g_{j} : X \to \R$ measurable functions with $\norm{g_{j}}_{L^{1,\infty}(X)} \leq a_{j}$ for every $j\in I$, where $(a_{j})_{j\in I}\subset (0, \infty)$ are positive numbers.

Then from \cite[Lemma 2.3]{MR0241685} we know
\begin{equation}
\label{lem:l1inf-quasi-triangle}
\norm[\Big]{ \sum_{j\in I} g_{j} }_{L^{1,\infty}(X)}
\leq 2
\sum_{j\in I} a_{j} \Big(\log \Big({a_{j}}^{-1}\sum_{j'\in I}{a_{j'}} \Big)+2\Big).
\end{equation}
The same argument shows that the spaces $L^{p,\infty}$ are $p$-convex for $p\in(0, 1)$. 
\begin{lemma}
Given a measure space $(X,\calB,\frakm)$, let $p\in (0,1)$,  $I$ be a countable set, and let $g_{j} : X \to [0,\infty)$ be measurable functions in $L^{p, \infty}(X)$ for every $j\in I$. Then
\begin{equation}
\label{lem:lpinf-quasi-triangle}
\norm[\Big]{\sum_{j\in I} g_{j}}_{L^{p,\infty}(X)}^{p}
\lesssim_{p}
\sum_{j\in I} \norm{ g_{j} }_{L^{p,\infty}(X)}^{p}.
\end{equation}
\end{lemma}
\begin{proof}
By scaling it suffices to show that if $\frakm\big({\Set{x\in X\given g_{j}(x) > s}}\big) \leq s^{-p}$ for all $s>0$ and $j\in I$ and $c_{j}\geq 0$ are numbers such that $\sum_{j \in I} c_{j}^{p} \leq 1$, then
\[
\frakm\big({\Set{x\in X\given \sum_{j\in I}c_{j} g_{j}(x) > s}}\big)
\lesssim_{p}
s^{-p} \quad \text{for all}\quad s>0.
\]
Without loss of generality we may assume $c_{j}>0$ for all $j\in I$.
Let
\[
u_{j} = (g_{j}-s/2) \one_{{\Set{x\in X\given g_{j}(x) > s/c_j}}},
\qquad
l_{j} = \min( g_{j}, s/2 ),
\qquad
m_{j} = g_{j} - u_{j} - l_{j}.
\]
Then
\[
\sum_{j\in I} c_{j} l_{j}
\leq
\frac{s}{2} \sum_{j\in I} c_{j}
\leq
\frac{s}{2} \Big( \sum_{j\in I} c_{j}^{p} \Big)^{1/p}
\leq
\frac{s}{2},
\]
\[
\frakm\Big(\bigcup_{j\in I}\supp u_{j}\Big)
\leq
\sum_{j\in I} \frakm\big({\Set{x\in X\given g_{j}(x) > s/c_j}}\big)
\leq
\sum_{j\in I} (s/c_{j})^{-p}
=
s^{-p} \sum_{j\in I} c_{j}^{p}
\leq
s^{-p},
\]
and
\begin{multline*}
\int_X \sum_{j\in I} c_{j} m_{j}(x)\dif \frakm(x)
=
\sum_{j\in I} c_{j} \int_{s/2}^{s/c_{j}} \frakm\big({\Set{x\in X\given m_{j}(x) > y }}\big) \dif y\\
\leq
\sum_{j\in I} c_{j} \int_{0}^{s/c_{j}} y^{-p} \dif y
=
\frac{1}{1-p} \sum_{j\in I} c_{j} (s/c_{j})^{1-p}
\leq
\frac{s^{1-p}}{1-p},
\end{multline*}
so that
\begin{align*}
\frakm \big(\Set{x\in X\given \sum_{j\in I}c_{j}g_{j}(x) > s }\big)
&\leq
s^{-p} + \frakm \big(\Set{x\in X\given \sum_{j\in I}c_{j}m_{j}(x) > s/2 }\big)\\
&\leq
s^{-p} + \frac{2}{s} \int_X \sum_{j\in I}c_{j}m_{j}(x)\dif\frakm(x)\\
&\leq
\Bigl( 1 + \frac{2}{1-p} \Bigr) s^{-p}.
\qedhere  
\end{align*}
\end{proof}

We will deduce Lemma~\ref{lem:jumps-control-variation} from a more general result for sequences that is also useful in the context of paraproducts \cite{MR2949622,arxiv:1812.09763} where variation and jump seminorms are defined differently.
For $p,q \in (0,\infty]$, $\rho \in (0,\infty)$, and a measurable function $F : X \times \N \to [0,\infty)$ we write
\[
\calJ^{p,q}_{\rho}(F) := \sup_{\lambda > 0} \norm{\lambda \calN_{\lambda}^{1/\rho}}_{L^{p,q}(X)},
\quad
\calN_{\lambda}(x) := \card{\Set{n \given F(x,n) \geq \lambda}}.
\]
\begin{lemma}
\label{lem:jumps-control-variation:difference}
For every $p\in (0, \infty]$ and $\rho \in (0,\infty)$ there exists $0<C_{p,\rho}<\infty$ such that the following holds.
Let $r \in (\rho, \infty)$.
Then for every $\sigma$-finite measure space $(X, \calB, \frakm)$ and every measurable function $F : X \times \N \to [0,\infty)$ we have the estimates
\begin{equation}
\label{eq:Vr<N:difference}
\norm{F}_{L^{p,\infty}(X; \ell^{r})}
\le C_{p,\rho}
\begin{cases}
\big(\frac{r}{r-\rho}\big)^{1/p} \calJ^{p,\infty}_{\rho}(F)
& \text{if } p<\rho,\\
\big(\frac{r}{r-\rho} \big(1 + \log \frac{r}{r-\rho} \big) \big)^{1/\rho} \calJ^{p,\infty}_{\rho}(F)
& \text{if } p=\rho \text{ and}\\
\big(\frac{r}{r-\rho} \big)^{1/\rho} \calJ^{p,p}_{\rho}(F)
& \text{if } p=\rho,\\
\big(\frac{r}{r-\rho}\big)^{1/\rho} \calJ^{p,\infty}_{\rho}(F)
& \text{if } p>\rho.
\end{cases}
\end{equation}
\end{lemma}

\begin{proof}[Proof of Lemma~\ref{lem:jumps-control-variation} assuming Lemma~\ref{lem:jumps-control-variation:difference}]
Since $\bbI$ is finite, for every $x$ the supremum in the definition \eqref{eq:192} of $r$-variation is assumed for some increasing sequence $(t_{x,j})_{j}$.
We may assume that this sequence depends measurably on $x$.
Let $F(x,j) := \norm{f(x,t_{x,j+1})-f(x,t_{x,j})}_{B}$ and continue this sequence by $0$ for those $j$ for which $t_{x,j+1}$ is not defined.
Then $\norm{f}_{L^{p,\infty}(X; V^{r}_{\bbI \to B})} = \norm{F}_{L^{p,\infty}(X; \ell^{r})}$ and $\calJ^{p,q}_{\rho}(F) \leq J^{p,q}_{\rho}(f)$.
\end{proof}

\begin{proof}[Proof of Lemma~\ref{lem:jumps-control-variation:difference}]
By monotonicity of $\ell^{r}$ norms it suffices to consider $\rho < r \leq 2\rho$.
By scaling we may replace the $p$-th power of the left-hand side of \eqref{eq:Vr<N} by
\[
\frakm \big(\Set{x\in X \given \norm{F(x,\cdot)}_{\ell^{r}} > 1}\big).
\]
Let
\[
A := \calJ^{p,\infty}_{\rho}(F) =
\sup_{\lambda>0} \norm[\big]{ \lambda \calN_{\lambda}^{1/\rho} }_{L^{p,\infty}(X)}.
\]
Note that
\begin{align*}
\frakm\big(\Set{x\in X \given \norm{F(x,\cdot)}_{\ell^{\infty}} \ge 1}\big)
&=\frakm\big(\Set{x\in X \given \calN_{1}(x) \ge 1}\big)\\
&\leq
\norm{ 1 \cdot N_{1}^{1/\rho} }_{L^{p,\infty}(X)}^{p}
\leq
A^{p}.
\end{align*}
Therefore, it remains to estimate the measure of the set
\[
X' := \Set{x\in X \given \norm{F(x,\cdot)}_{\ell^{r}} > 1 > \norm{F(x,\cdot)}_{\ell^{\infty}} }.
\]
For $x\in X'$ we have
\[
\norm{F(x, \cdot)}_{\ell^{r}}^r
\leq
\sum_{j<0} 2^{(j+1) r} \calN_{2^{j}}(x)
\]
which yields
\begin{equation}
\label{eq:194}
\frakm(X')
\le
\frakm\big(\Set{x\in X \given \sum_{j<0} 2^{jr} \calN_{2^{j}}(x) > 2^{-2\rho}}\big)
\lesssim_{p,\rho}
\norm[\Big]{ \sum_{j<0} 2^{jr} \calN_{2^{j}} }_{L^{p/\rho,\infty}(X)}^{p/\rho}.
\end{equation}
We distinguish three cases to estimate \eqref{eq:194}.
Suppose first $\rho<p$.
Then, since $L^{p/\rho,\infty}(X)$ admits an equivalent subadditive norm, we get
\begin{align*}
\eqref{eq:194}
&\lesssim_{p, \rho}
\Big( \sum_{j< 0} 2^{jr} \norm[\big]{ \calN_{2^{j}} }_{L^{p/\rho,\infty}(X)} \Big)^{p/\rho}\\
&=
\Big( \sum_{j< 0} 2^{j(r-\rho)} \norm[\big]{ 2^{j} \calN_{2^{j}}^{1/\rho} }_{L^{p,\infty}(X)}^{\rho} \Big)^{p/\rho}\\
&\leq
A^{p} \Big( \sum_{j\leq 0} 2^{j(r-\rho)} \Big)^{p/\rho}\\
&=
A^{p} ( 1-2^{-(r-\rho)} )^{-p/\rho}.
\end{align*}
Suppose now $p<\rho$.
Then by \eqref{lem:lpinf-quasi-triangle} we have
\begin{align*}
\eqref{eq:194}
&\lesssim_{p,\rho}
\sum_{j\leq 0} \norm[\big]{2^{jr} \calN_{2^{j}} }_{L^{p/\rho,\infty}(X)}^{p/\rho}\\
&=
\sum_{j\leq 0} 2^{j(r-\rho)p/\rho} \norm[\big]{ 2^{j} \calN_{2^{j}}^{1/\rho} }_{L^{p,\infty}(X)}^{p}\\
&\leq
A^{p} \sum_{j\leq 0} 2^{j(r-\rho)p/\rho}\\
&=
A^{p} (1-2^{-(r-\rho)p/\rho})^{-1}.
\end{align*}
In the case $p=\rho$ we have
\[
\norm{\calN_{2^{j}}}_{L^{1,\infty}(X)}
=
2^{-j\rho} \norm{2^{j} \calN_{2^{j}}^{1/\rho}}_{L^{p,\infty}(X)}^{p}
\leq
2^{-j\rho} A^{p},
\]
and using \eqref{lem:l1inf-quasi-triangle} with $a_{j} = A^{p} 2^{j(r-\rho)}$ we obtain
\begin{align*}
\eqref{eq:194}
&\leq
  2 \sum_{j\leq 0} a_{j} \Big( \log \Big(a_{j}^{-1}\sum_{j'\leq 0} a_{j'}\Big) + 2 \Big)\\
&=
2 A^{p} \sum_{j\leq 0} 2^{j(r-\rho)} \Big( \log \Big(2^{-j(r-\rho)}\sum_{j'\leq 0} 2^{j'(r-\rho)}\Big) + 2 \Big)\\
&\lesssim
A^{p} \sum_{j\leq 0} 2^{j(r-\rho)} \bigl(-j(r-\rho) - \log (r-\rho) + 2 \bigr)\\
&\lesssim
A^{p} (r-\rho)^{-1} ( 1 - \log (r-\rho) ).
\end{align*}
Alternatively, still in the case $p=\rho$, we can estimate
\begin{align*}
\eqref{eq:194}
&
\leq
\norm[\Big]{ \sum_{j<0} 2^{jr} \calN_{2^{j}} }_{L^{1}(X)}\\
&\leq
\sum_{j<0} 2^{j(r-\rho)} \norm[\Big]{ 2^{j} \calN_{2^{j}}^{1/\rho} }_{L^{\rho}(X)}^{\rho}\\
&\lesssim
(r-\rho)^{-1} \sup_{\lambda > 0} \norm[\Big]{ \lambda \calN_{\lambda}^{1/\rho} }_{L^{\rho}(X)}^{\rho}.
\qedhere
\end{align*}
\end{proof}

\section{Endpoint L\'epingle inequality for martingales}
\label{sec:martingale}
Let $(X, \calB, \frakm)$ be a $\sigma$-finite measure space and $\bbI$ a totally ordered set.
A sequence of sub-$\sigma$-algebras $(\calG_t)_{t\in \bbI}$ of $\calB$ is called a \emph{filtration} if it is increasing and the measure $\frakm$ is $\sigma$-finite on each $\calG_t$.
Let $B$ be a Banach space.
A $B$-valued \emph{martingale} adapted to a filtration $(\calG_t)_{t\in \bbI}$ is a family of functions $\frakf=(\frakf_t)_{t\in \bbI} \subset L_{\mathrm{loc}}^1(X, \calB,\frakm;B)$ such that $\frakf_{t'}=\EE{\frakf_t \given \calG_{t'}}$ for every $t',t\in\bbI$ with $t'\leq t$, where $\EE{\cdot \given \calG}$ denotes the conditional expectation with respect to a sub-$\sigma$-algebra $\calG\subseteq\calB$.

We recall from \cite[Theorem 10.59]{MR3617459} that a Banach space $B$ has \emph{martingale cotype $\rho\in [2, \infty)$} if and only if for any $B$-valued martingale $\frakf= (\frakf_{n})_{n\in \N}$ the ``$\rho$-square function''
\[
S_{\rho}\frakf := \Big( \sum_{n>0} \norm{\frakf_{n}-\frakf_{n-1}}_B^{\rho} \Big)^{1/\rho}
\]
satisfies the estimates
\begin{equation}
\label{eq:martingale-square-max}
\norm{S_{\rho}\frakf}_{L^p(X)} \lesssim_{p} \norm{\frakf_{\star}}_{L^p(X)},
\quad p\in [1,\infty),
\end{equation}
where $\frakf_{\star}(x) := \sup_{n\in\N}\norm{\frakf_n(x)}_B$ is the martingale maximal function and the implicit constant does not depend on $\frakf$.
By Doob's inequality, see e.g.\ \cite[Corollary 1.28]{MR3617459}, we know that
\begin{equation}
\label{eq:doob-max}
\norm{\frakf_{\star}}_{L^p(X)}
\leq p' \sup_{n\in \N}\norm{\frakf_n}_{L^p(X;B)},
\quad p\in(1,\infty].
\end{equation}
A Banach space has martingale cotype $\rho$ for some $\rho \in [2,\infty)$ if and only if it is uniformly convex, see \cite[Chapter 10]{MR3617459}.

Now we are in a position to formulate the quantitative version of the endpoint L\'epingle inequality for martingales.
\begin{theorem}
\label{thm:endpoint-lepingle}
Given $p\in(1,\infty)$ and $\rho\in[2, \infty)$, let $B$ be a Banach space and $(X, \calB, \frakm)$ a $\sigma$-finite measure space.
Suppose that the inequality
\begin{equation}
\label{eq:martingale-square-bound}
\norm{S_{\rho} \frakf}_{L^{p}(X)} \leq A_{p,\rho,B} \sup_{n\in\N} \norm{\frakf_{n}}_{L^{p}(X;B)}
\end{equation}
holds for arbitrary martingales $(\frakf_{n})_{n\in\N}$ with values in $B$.
Then for every finite totally ordered set $\bbI$ and every martingale $\frakf=(\frakf_{t})_{t\in \bbI} : X \to B$ indexed by $\bbI$ with values in $B$ we have
\begin{equation}
\label{eq:endpoint-lepingle}
J^{p}_{\rho}(\frakf) \simeq
[L^{\infty}(X;V_{\bbI\to B}^{\infty}),L^{p/\rho}(X;V_{\bbI \to B}^{1})]_{1/\rho,\infty}(\frakf)
\leq 3 A_{p,\rho,B}
\sup_{t\in \bbI}\norm{\frakf_{t}}_{L^p(X;B)}.
\end{equation}
\end{theorem}
Theorem~\ref{thm:endpoint-lepingle:simple} follows from Theorem~\ref{thm:endpoint-lepingle} since for a Banach space $B$ with martingale cotype $\rho$ the estimate \eqref{eq:martingale-square-bound} holds with some finite constant $A_{p,\rho,B}<\infty$ in view of \eqref{eq:martingale-square-max} and \eqref{eq:doob-max}.

Our proof of Theorem~\ref{thm:endpoint-lepingle} is even simpler than the one presented in \cite[Lemma 2.2]{MR933985} for $p=2$.
At the endpoint $p=1$, assuming the weak type analogue of \eqref{eq:martingale-square-bound}, it yields the weak type estimate
\begin{equation}
\label{eq:L1-lepingle}
J^{1,\infty}_{\rho}(f) \simeq
[L^{\infty}(X;V^{\infty}_{\bbI \to B}),L^{1/\rho,\infty}(X;V^{1}_{\bbI \to B})]_{1/\rho,\infty}(\frakf)
\lesssim
\sup_{t\in \bbI}\norm{\frakf_{t}}_{L^1(X;B)}.
\end{equation}

\begin{proof}
By homogeneity we may assume $\sup_{t\in \bbI}\norm{\frakf_{t}}_{L^p(X;B)}=1$.
Let $\lambda>0$ and construct stopping times $t_{0},t_{1},\dotsc : X \to \bbI \cup \Set{+\infty}$ associated to $\lambda$-jumps as in \eqref{eq:lambda-jump-stopping-time} (note that they are indeed stopping times in the stochastic sense).
Split $\frakf_{t}=\frakf_{t}^{0}+\frakf_{t}^{1}$ with
\[
\frakf_{t}^{1}(x) := \sum_{k\geq 0} \one_{[t_{k}(x),  t_{k+1}(x))}(t) \frakf_{t_{k}(x)}(x).
\]
Then by construction $\norm{\frakf_{t}^{0}(x)}_{B}\leq\lambda$ for all $t\in\bbI$ and $x\in X$.
On the other hand,
\begin{align*}
V^{1}(\frakf_{t}^{1}(x) : t\in\bbI)
&\le
\sum_{\substack{k \in \N :\\ k\geq 1 \text{ and } t_{k}(x) < +\infty}} \underbrace{\norm{\frakf_{t_{k}(x)}(x)-\frakf_{t_{k-1}(x)}(x)}_{B}}_{\geq \lambda}\\
&\leq
\lambda^{1-\rho} \sum_{\substack{k \in \N :\\ k\geq 1 \text{ and } t_{k}(x) < +\infty}} \norm{\frakf_{t_{k}(x)}(x)-\frakf_{t_{k-1}(x)}(x)}_{B}^{\rho}\\
&=
\lambda^{1-\rho} (S_{\rho}\tilde\frakf)^{\rho},
\end{align*}
where
\[
\tilde\frakf_{k}(x) :=
\begin{cases}
\frakf_{t_{k}(x)}(x) & \text{if } t_{k}(x) < + \infty,\\
\frakf_{\max\Set{t_{k'}(x) \given k'\in\N, t_{k'}(x)<\infty}}(x) & \text{if } t_{k}(x) = + \infty
\end{cases}
\]
is the stopped martingale.
Thus
\begin{align*}
\norm{ V^{1}(\frakf_{t}^{1} : t\in\bbI) }_{L^{p/\rho}(X)}
&\leq
\lambda^{1-\rho} \norm{ (S_{\rho}\tilde\frakf)^{\rho} }_{L^{p/\rho}(X)}\\
&=
\lambda^{1-\rho} \norm{ S_{\rho}\tilde\frakf }_{L^{p}(X)}^{\rho}\\
&\leq
\lambda^{1-\rho} A_{p,\rho,B}^{\rho} \sup_{k\in\N} \norm{\tilde\frakf_{k}}_{L^{p}(X;B)}\\
&\leq
\lambda^{1-\rho} A_{p,\rho,B}^{\rho},
\end{align*}
so that
\begin{multline*}
K(\lambda^{\rho} A_{p,\rho,B}^{-\rho}, \frakf; L^{\infty}(X;V_{\bbI\to B}^{\infty}),L^{p/\rho}(X;V_{\bbI \to B}^{1}))\\
\leq
\norm{\frakf^{0}}_{L^{\infty}(X;V_{\bbI\to B}^{\infty})} + \lambda^{\rho} A_{p,\rho,B}^{-\rho} \norm{\frakf^{1}}_{L^{p/\rho}(X;V_{\bbI \to B}^{1})}
\leq
3 \lambda.
\end{multline*}
Since $\lambda>0$ was arbitrary we obtain
\[
K(\lambda, \frakf; L^{\infty}(X;V_{\bbI\to B}^{\infty}),L^{p/\rho}(X;V_{\bbI \to B}^{1}))
\leq
3 A_{p,\rho,B} \lambda^{1/\rho}
\]
for all $\lambda \in (0,\infty)$.
The conclusion \eqref{eq:endpoint-lepingle} follows by definition of real interpolation spaces.
\end{proof}

\subsection{Doubly stochastic operators}
In this section we prove Theorem~\ref{thm:markov-jump}.
By the monotone convergence theorem it suffices to consider $n$ in a finite subset $\bbI \subset \N$.
By Rota's dilation theorem (proved in \cite{MR0133847} on probability spaces and \cite{MR0190757} on $\sigma$-finite measure spaces) there exists a measure space $(\Omega,\tilde\calB,\tilde\frakm)$, a sub-$\sigma$-algebra $\calG \subseteq \tilde\calB$, a measure space isomorphism
\[
\iota : (\Omega, \calG, \tilde\frakm) \to (X, \calB, \frakm),
\]
and a decreasing sequence of sub-$\sigma$-algebras
\[
\tilde{\calB} \supseteq \calG_{0} \supseteq \calG_{1} \supseteq \dotsb
\]
such that for every $f\in L^{1}(X) + L^{\infty}(X)$ and every $n\in\N$ we have
\begin{equation}
\label{eq:rota-representation}
(Q^{*})^{n}Q^{n}f
=
S \big(\EE{f \circ \iota \given \calG_{n}}\big),
\end{equation}
where the operator $S : L^{p}(\Omega) \to L^{p}(X)$ is characterized by $Sf \circ \iota = \EE{f \given \calG}$.
The operators on both sides of \eqref{eq:rota-representation} are positive contractions on all spaces $L^{p}(X)$ for $1\leq p \leq \infty$, and therefore their algebraic tensor products with $\id_{B}$ extend uniquely to contractions on the Bochner spaces $L^{p}(X;B)$ for all $1 \leq p \leq \infty$ by \cite[Proposition 1.6]{MR3617459}, that also have to coincide.

Let $B$ be a Banach space with martingale cotype $2 \leq \rho < \infty$ and $1 < p < \infty$.
By Theorem~\ref{thm:endpoint-lepingle} for every function $f$ in the Bochner space $L^{p}(X;B)$ we have the inequality
\[
J^{p}_{\rho}((\EE{f \circ \iota \given \calG_{n}})_{n\in \bbI} : \Omega  \to B)
\lesssim_{p,\rho,B}
\norm{f}_{L^{p}(X;B)}.
\]
By Lemma~\ref{lem:pisier-xu-space} with $\theta = \max(1/p, 1/\rho)$ the left-hand side of this estimate is equivalent to
\[
[L^{\infty}(\Omega;V^{\infty}_{\bbI\to B}), L^{\theta p}(\Omega;V^{\theta\rho}_{\bbI \to B})]_{\theta,\infty}((\EE{f \circ \iota \given \calG_{n}})_{n \in \bbI}).
\]
The operator $S : L^{p}(\Omega) \to L^{p}(X)$ is a contraction for every $1 \leq p \leq \infty$.
Since the operator $S$ is positive, by \cite[Proposition 1.6]{MR3617459} for every Banach space $\tilde B$ the algebraic tensor product operator $S \otimes \id_{\tilde B}$ extends to a contraction
\[
L^{p}(\Omega;\tilde B) \to L^{p}(X;\tilde B)
\]
that will be again denoted by $S$.
Applying this with $\tilde B = V^{\infty}_{\bbI\to B}$ and $\tilde B = V^{\theta\rho}_{\bbI\to B}$ and using the Marcinkiewicz interpolation theorem \cite[Theorem 3.1.2]{MR0482275} we see that the operator $S$ extends to a contraction
\[
[L^{\infty}(\Omega;V^{\infty}_{\bbI\to B}), L^{\theta p}(\Omega;V^{\theta\rho}_{\bbI \to B})]_{\theta,\infty}
\to
[L^{\infty}(X;V^{\infty}_{\bbI\to B}), L^{\theta p}(X;V^{\theta\rho}_{\bbI \to B})]_{\theta,\infty}.
\]
By Lemma~\ref{lem:pisier-xu-space} the norm on the interpolation space on the right-hand side is equivalent to $J^{p}_{\rho}(f : X \times \bbI \to B)$, and the conclusion follows from \eqref{eq:rota-representation}.

\section{Sampling for  Fourier multipliers}
\subsection{Interpolation between Bochner spaces}
In Section~\ref{sec:transference-interpolation} we will have to assemble estimates in interpolation spaces on congruence classes modulo $q$ into estimates on all of $\Z^{d}$.
The following result will allow us to do this in an abstract setting.
\begin{lemma}
\label{lem:weak-lp-cong-class-split}
For every $p\in(1,\infty)$ and $\theta\in(0, 1)$ such that $1\leq \theta p$ there exists a constant $0<C_{p, \theta} < \infty$ such that the following holds.
Let $(X, \calB, \frakm)$ be a $\sigma$-finite measure space, $(A_{0},A_{1})$ a compatible couple of Banach spaces, $f:X \to A_{0}+A_{1}$ a measurable function, and $X=\bigcup_{j\in I} X_{j}$ a countable measurable partition.
Then
\begin{align*}
\MoveEqLeft{}
[L^{\infty}(X;A_{0}),L^{\theta p}(X;A_{1})]_{\theta,\infty}(f)\\
&\leq C_{p,\theta}
\bigg( \sum_{j\in I} \Big( [L^{\infty}(X_{j};A_{0}),L^{\theta p}(X_{j};A_{1}))]_{\theta,\infty}(f) \Big)^{p} \bigg)^{1/p}. 
\end{align*}
\end{lemma}
In the case $A_{0}=A_{1}=A$ Lemma~\ref{lem:weak-lp-cong-class-split} follows readily from the fact that
\[
[L^{\infty}(A_{0}),L^{\theta p}(A_{1})]_{\theta,\infty} = L^{p,\infty}(A)
\]
and the description of the Lorentz space $L^{p,\infty}(A)$ in terms of superlevel sets.
For general Banach spaces we will use an explicit description of the $K$-functional between Bochner spaces going back to \cite[Theorem 5]{MR1213125} in the following form.
\begin{theorem}[{\cite[Remark 8.61]{MR3617459}}]
\label{thm:K-vector}
For every $\rho\in[1, \infty)$ there exists a constant $0<C_{\rho} < \infty$ such that the following holds.
Let $(X, \calB, \frakm)$ be a $\sigma$-finite measure space and $(A_{0},A_{1})$ a compatible couple of Banach spaces.
Then for every function $f\in L^{\rho}(X;A_{0})+L^{\infty}(X;A_{1})$ and every $t>0$ we have
\begin{align}
\label{eq:K-vector}
\begin{split}
\MoveEqLeft{}
C_{\rho}^{-1}K(t,f;L^{\rho}(X; A_{0}),L^{\infty}(X; A_{1}))^{\rho}\\
&\le
\sup_{\psi> 0 :\ \norm{\psi}_{L^{\rho}(X)}= t} \int_X
K(\psi(x),f(x);A_{0},A_{1})^{\rho} \dif\frakm(x)\\
&\le C_{\rho}K(t,f;L^{\rho}(X;A_{0}),L^{\infty}(X;A_{1}))^{\rho},
\end{split}
\end{align}
where the supremum is taken over all strictly positive measurable and $\rho$-integrable functions $\psi : X \to (0,\infty)$.
\end{theorem}
In Pisier's book \cite[Remark 8.61]{MR3617459} this result is formulated with a supremum over non-negative functions $\psi$ such that $\norm{\psi}_{L^{\rho}(X)}\le t$, however it is easily seen that
\begin{align*}
L(f)& :=\sup_{\psi> 0 :\ \norm{\psi}_{L^{\rho}(X)}= t} \int_X
K(\psi(x),f(x);A_{0},A_{1})^{\rho} \dif\frakm(x)\\
&=\sup_{\psi\ge0 :\ \norm{\psi}_{L^{\rho}(X)}\le t} \int_X
K(\psi(x),f(x);A_{0},A_{1})^{\rho} \dif\frakm(x)
=: R(f).
\end{align*}
Indeed, it suffices to show that $R(f)\le L(f)$.
Let $\epsilon>0$ and consider a measurable function $\psi : X \to [0,\infty)$ such that $\norm{\psi}_{L^{\rho}(X)}\le t$ and
\begin{align*}
R(f) < \epsilon+\int_X
K(\psi(x),f(x);A_{0},A_{1})^{\rho} \dif\frakm(x).
\end{align*}
Now we take a positive measurable function $\phi : X \to (0,\infty)$ such that $\norm{\phi}_{L^{\rho}(X)}= t$ and $\phi(x)\ge \psi(x)/(1+\epsilon)$ for every $x\in X$.
By \eqref{eq:206} and monotonicity of the $K$-functional we obtain
\begin{align*}
R(f)
&< \epsilon+\int_X K(\psi(x),f(x);A_{0},A_{1})^{\rho} \dif\frakm(x)\\
&\leq \epsilon+(1+\epsilon)^{\rho}\int_X K(\psi(x)/(1+\epsilon),f(x);A_{0},A_{1})^{\rho} \dif\frakm(x)\\
&\leq \epsilon+(1+\epsilon)^{\rho}\int_X K(\phi(x),f(x);A_{0},A_{1})^{\rho} \dif\frakm(x)\\
&\leq \epsilon+(1+\epsilon)^{\rho} L(f).
\end{align*}
This proves $R(f)\le L(f)$, since $\epsilon$ is arbitrary.
\begin{proof}[Proof of Lemma~\ref{lem:weak-lp-cong-class-split}]
By \eqref{eq:K-vector} with $\rho=\theta p$ we have
\[
[L^{\infty}(X;A_{0}),L^{\theta p}(X;A_{1})]_{\theta,\infty}(f)^{p}
=
[L^{\theta p}(X;A_{1}),L^{\infty}(X;A_{0})]_{1-\theta,\infty}(f)^{p}
\]
\[
=
\big( \sup_{t>0} t^{\theta-1} K(t,f;L^{\theta p}(X;A_{1}),L^{\infty}(X;A_{0})) \big)^{p}
\]
\[
\simeq_{\rho}
\sup_{\psi > 0} \bigg(\int_X\psi(x)^{\theta p}\dif\frakm(x) \bigg)^{1-1/\theta}\bigg( \int_X
{\underbrace{K(\psi(x),f(x);A_{1},A_{0})}_{=:K_{\psi}(x)}}^{\theta
p} \dif\frakm(x)\bigg)^{1/\theta}
\]
\begin{align}
\label{eq:vector-ellp-int-sp}
=\sup_{\psi > 0}\norm{\psi^{\theta p}}_{L^{\theta p}(X)}^{\theta
p(1-1/\theta)}\big(\norm{K_{\psi}}_{L^{\theta p}(X)}^{\theta p}\big)^{1/\theta}.
\end{align}
We may assume that each $X_{j}$ has non-zero measure.
By H\"older's inequality with exponent $1/\theta$ we obtain the estimate
\begin{align*}
\eqref{eq:vector-ellp-int-sp}&=
\sup_{\psi> 0} \bigg( \sum_{j\in I} \norm{\psi^{\theta p}}_{L^{\theta p}(X_j)}^{\theta
p} \bigg)^{1-1/\theta}
\bigg( \sum_{j\in I} \norm{\psi^{\theta p}}_{L^{\theta p}(X_j)}^{\theta
p(1-\theta)} \, \norm{\psi^{\theta p}}_{L^{\theta p}(X_j)}^{\theta
p(\theta-1)} \, \norm{K_{\psi}}_{L^{\theta p}(X_j)}^{\theta p} \bigg)^{1/\theta}\\
&\leq
\sup_{\psi> 0} \sum_{j\in I} \big(\norm{\psi^{\theta p}}_{L^{\theta p}(X_j)}^{\theta
p} \big)^{1-1/\theta}\, \big( \norm{K_{\psi}}_{L^{\theta p}(X_j)}^{\theta p}\big)^{1/\theta}\\
&\simeq_{\rho}
\sum_{j\in I} \Big( [L^{\theta p}(X_{j};A_{0}),L^{\infty}(X_{j};A_{1})]_{\theta,\infty}(f) \Big)^{p},
\end{align*}
where we have used \eqref{eq:vector-ellp-int-sp} on each $X_{j}$ in the last line.
\end{proof}

\subsection{Sampling and continuation of band limited functions}
\label{sec:sampling}
Let $B$ be a finite-dimensional Banach space.
Consider the extension operator $\calE$ (see \cite[formula (2.2)]{MR1888798}) that maps a vector-valued sequence $f : \Z^{d} \to B$ to the vector-valued function on $\R^{d}$ defined by
\[
\calE f(x) := \sum_{n\in\Z^{d}} f(n) \Psi(x-n),
\quad \text{where}\quad
\Psi(x) = \prod_{i=1}^{d} \bigg( \frac{\sin(\pi x_{i})}{\pi x_{i}} \bigg)^{2}.
\]
Note that $\calE f(n)=f(n)$ for all $n\in\Z^d$, since $\Psi(x)=0$ for all $x\in\Z^d\setminus\{0\}$.

Consider also the restriction operator that maps a vector-valued function $F :\R^{d} \to B$ to the vector-valued sequence on $\Z^{d}$ defined by
\[
\calR F(n) := \int_{\R^{d}} F(y) \Phi(n-y) \dif y,
\]
where $\Phi$ is a Schwartz function such that
\[
\widehat\Phi(\xi)=
\begin{cases}
1, & \text{ if } \abs{\xi}_{\infty}\leq 1,\\
0, & \text{ if } \abs{\xi}_{\infty}\ge 2,
\end{cases}
\]
where $\widehat \ $ denotes the Fourier transform on $\R^d$.

It was proved in \cite[Lemma 2.1]{MR1888798} that $\calR\calE=\id$ and there exists an absolute constant $0<C_{d}<\infty$ (independent of the finite-dimensional Banach space $B$) such that for every $p\in [1,\infty]$ the norms of the operators
\begin{equation}
\label{eq:Fourier-ext-restr}
\calE : \ell^{p}(\Z^d;B) \to L^{p}(\R^d;B),
\qquad
\calR : L^{p}(\R^d;B) \to \ell^{p}(\Z^d;B)
\end{equation}
are bounded by $C_d$.
This was used to deduce the following result for periodic Fourier multipliers.
\begin{proposition}[{\cite[Corollary 2.1]{MR1888798}}]
\label{prop:msw-mult}
There exists an absolute constant $0<C<\infty$ such that the following holds.
Let $p \in [1,\infty]$, $q\in\N$ be a positive integer, and let $B_1, B_2$ be finite-dimensional Banach spaces.
Let $m : \R^{d} \to L(B_1, B_2)$ be a bounded operator-valued function supported on $[-1/2,1/2]^{d}/q$ and denote the associated Fourier multiplier operator over $\R^{d}$ by $T$.
Let $m^{q}_{\mathrm{per}}$ be as in \eqref{eq:m_per} and denote the associated Fourier multiplier operator over $\Z^{d}$ by $T^{q}_{\dis}$.
Then
\[
\norm{T^{q}_{\dis}}_{\ell^{p}(\Z^{d};B_1)\to\ell^{p}(\Z^{d};B_2)}\le
C\norm{T}_{L^{p}(\R^{d};B_1)\to L^{p}(\R^{d};B_2)}.
\]
\end{proposition}

\subsection{Sampling in interpolation spaces}
\label{sec:transference-interpolation}
Proposition~\ref{prop:msw-mult} cannot be applied to the jump space \eqref{eq:jump-space} because it is not a Bochner space.
However, by \eqref{eq:197} it coincides with an interpolation space between Bochner spaces. Therefore, 
 a version of Proposition~\ref{prop:msw-mult} that involves interpolation spaces will be proved.

\begin{proposition}
\label{prop:msw-mult-interpolation}
For every $d\geq 1$, $p\in(1, \infty)$ and $\theta\in(0, 1)$ such that $1\leq p\theta$ there exists a constant $0<C_{p,\theta,d}<\infty$ such that the following holds.
Let $q\in\N$ be a positive integer, let $A_{0}, A_{1}, B$ be finite-dimensional Banach spaces and assume that $(A_0, A_1)$ is a compatible pair.
Let $m:\R^{d}\to L(B,A_0+A_1)$ be a bounded operator-valued function supported on $[-1/2,1/2]^{d}/q$ and denote the associated Fourier multiplier operator by $T$.
Define the discrete multiplier operator $T^{q}_{\dis}$ as in Proposition~\ref{prop:msw-mult}.
Then
\begin{align}
\label{eq:22}
\begin{split}
\norm{T^{q}_{\dis}}_{\ell^{p}(\Z^{d};B)\to[\ell^{\infty}(\Z^{d};A_0),\ell^{\theta p}(\Z^{d};A_1)]_{\theta,\infty}} \le
C_{p,\theta,d} \norm{T}_{L^{p}(\R^{d};B)\to[L^{\infty}(\R^{d};A_0),L^{\theta p}(\R^{d};A_1)]_{\theta,\infty}}.
\end{split}
\end{align}
\end{proposition}

\begin{proof}[Proof of Proposition~\ref{prop:msw-mult-interpolation}]
Let $\tau_rg(x)=g(x+r)$ and $\delta_qg(x)=g(qx)$ denote the translation and the dilation operator, respectively.
Partitioning $\Z^d$ into congruence classes modulo $q$ and using Lemma~\ref{lem:weak-lp-cong-class-split} we obtain
\begin{multline}
\label{eq:200}
[\ell^{\infty}(\Z^{d};A_0),\ell^{\theta p}(\Z^{d};A_1)]_{\theta,\infty}(T^{q}_{\dis}f)^p\\
\lesssim_{p,\theta}
\sum_{r\in\N_q^d} \Big( [\ell^{\infty}(q\Z^{d}+r;A_0),\ell^{\theta p}(q\Z^{d}+r;A_1)]_{\theta,\infty}(T^{q}_{\dis}f) \Big)^{p},
\end{multline}
where $\N_q:=\{1, \ldots, q\}$. Let $K(x) := \FT^{-1}(m)(x)$ for every $x\in\R^d$, where $\FT^{-1}$ is the inverse Fourier transform on $\R^d$. Recall from \cite{MR1888798} that the kernel of $T^{q}_{\dis}$ is given by the formula
\[
\FT^{-1}(m^{q}_{\mathrm{per}})(x)
=
\begin{cases}
q^{d} K(x), & x\in q\Z^{d},\\
0, & x\in\Z^{d}\setminus q\Z^{d}.
\end{cases}
\]
Let
$ T^q: L^{p}(\R^{d};B)\to[L^{\infty}(\R^{d};A_0),L^{\theta p}(\R^{d};A_1)]_{\theta,\infty}$ be the operator corresponding to the multiplier $m(\xi/q)$, which is supported in $[-1/2, 1/2]^d$.
The kernel of $T^q$ is $q^dK(qx)$ for $x\in\R^d$.
Let us define, as in \cite{MR1888798}, the discrete counterpart of $T^q$ by setting
\[
[T^q]_{\dis}f(x) := \sum_{y\in\Z^d}f(x-y)q^dK(qy).
\]
Then for every $x\in\Z^d$ we have
\begin{align}
\label{eq:202}
\begin{split}
T^{q}_{\dis}f(qx)
&=
\sum_{y\in\Z^d}f(qx-y)\FT^{-1}(m^{q}_{\mathrm{per}})(y)\\
&=
\sum_{y\in\Z^d}\delta_qf(x-y)q^dK(qy)=[T^q]_{\dis}(\delta_qf)(x).
\end{split}
\end{align}
Let $\tilde T^q$ be the Fourier multiplier operator whose multiplier
is
\[
\tilde m(\xi/q)=\sum_{l\in\Z^d: \abs{l}_{\infty}\le1}m((\xi+l)/q).
\]
Then by our assumptions
\begin{align}
\label{eq:203}
\begin{split}
\MoveEqLeft{}
[L^{\infty}(\R^{d};A_0),L^{\theta p}(\R^{d}; A_1)]_{\theta,\infty}(\tilde T^q(f))\\
&\leq 3^{d} \norm{T}_{L^{p}(\R^{d};B)\to[L^{\infty}(\R^{d};A_0),L^{\theta p}(\R^{d};A_1)]_{\theta,\infty}} \norm{f}_{\ell^p(\Z^d;B)}.
  \end{split}
\end{align}
Moreover, by \cite[formula (2.5)]{MR1888798} we have
\begin{align*}
\calE([T^q]_{\dis}(f))(x)=\tilde
T^q(\calE f)(x), \quad \text{ for } \quad x\in\R^d.
\end{align*}
Thus, in view of $\calR\calE=\id$, we get
\begin{align}
\label{eq:204}
[T^q]_{\dis}(f)(x)
=
\calR(\tilde T^q(\calE f))(x),
\quad \text{ for } \quad x\in\Z^d.
\end{align}

Also, by \eqref{eq:Fourier-ext-restr} and the Marcinkiewicz interpolation theorem \cite[Theorem 3.1.2]{MR0482275} we have
\begin{equation}
\label{eq:Fourier-restr-interpolation}
\calR : [L^{\infty}(\R^d;A_{0}),L^{\theta p}(\R^d;A_{1})]_{\theta,\infty}
\to
[\ell^{\infty}(\Z^d;A_{0}),\ell^{\theta p}(\Z^d;A_{1})]_{\theta,\infty}
\end{equation}
with a bound which does not depend on $A_{0},A_{1}$.

Now we can estimate the right-hand side of \eqref{eq:200} by
\begin{align*}
\eqref{eq:200}&=\sum_{r\in\N_q^d} [\ell^{\infty}(\Z^{d};A_0),\ell^{\theta p}(\Z^{d};A_1)]_{\theta,\infty}\big([T^{q}]_{\dis}(\delta_q(\tau_rf))\big)^{p}
&& \text{by \eqref{eq:202}}\\
&=
\sum_{r\in\N_q^d}[\ell^{\infty}(\Z^{d};A_0),\ell^{\theta p}(\Z^{d};A_1)]_{\theta,\infty}\big(\calR\big(\tilde T^{q}\calE(\delta_q(\tau_rf))\big)\big)^p
&& \text{by \eqref{eq:204}}\\
&\lesssim
\sum_{r\in\N_q^d}[L^{\infty}(\R^{d};A_0),L^{\theta p}(\R^{d};A_1)]_{\theta,\infty}\big(\tilde T^{q}\calE(\delta_q(\tau_rf))\big)^p
&& \text{by \eqref{eq:Fourier-restr-interpolation}}\\
&\lesssim \sum_{r\in\N_q^d}
\norm{\calE(\delta_q(\tau_rf))}_{L^p(\R^d;B)}^p
&& \text{by \eqref{eq:203}}\\
&\lesssim \sum_{r\in\N_q^d} \norm{\delta_q(\tau_rf))}_{\ell^p(\Z^d;B)}^p
&& \text{by \eqref{eq:Fourier-ext-restr}}\\
& = \norm{f}_{\ell^p(\Z^d;B)}^p.
\end{align*}
The proof is completed.
\end{proof}

\begin{proof}[Proof of Theorem~\ref{thm:msw-mult-jumps}]
By the monotone convergence theorem it suffices to consider finite sets $\bbI$.
Let $\theta=\min(1/p,1/\rho)$ and represent the jump space as an interpolation space using Lemma~\ref{lem:pisier-xu-space}.
Apply Proposition~\ref{prop:msw-mult-interpolation} with $A_{0} = V^{\infty}_{\bbI \to \C}$, $A_{1} = V^{\theta\rho}_{\bbI \to \C}$, and $B=\C$.
\end{proof}

\printbibliography
\end{document}